\documentclass[11pt]{article}
\usepackage{amsmath, amsthm}
\usepackage{latexsym}
\usepackage{graphicx, psfrag}
\makeatletter
\newfont{\footsc}{cmcsc10 at 8truept}
\newfont{\footbf}{cmbx10 at 8truept}
\newfont{\footrm}{cmr10 at 10truept}
\makeatother
\newtheorem{theorem}{\bf Theorem}
\newtheorem{proposition}{\bf Proposition}
\newtheorem{lemma}{\bf Lemma}
\newtheorem{conjecture}{\bf Conjecture}

\begin{document}
\title{Concave Renewal Functions Do Not Imply DFR Inter-Renewal Times}

\author{Yaming Yu\\
\small Department of Statistics\\[-0.8ex]
\small University of California\\[-0.8ex] 
\small Irvine, CA 92697, USA\\[-0.8ex]
\small \texttt{yamingy@uci.edu}}

\date{}
\maketitle

\begin{abstract}
Brown (1980, 1981) proved that the renewal function is concave if the inter-renewal distribution is DFR (decreasing failure rate), and conjectured the converse.  This note settles Brown's conjecture with a class of counter-examples.  We also give a short proof of Shanthikumar's (1988) result that the DFR property is closed under geometric compounding.


\end{abstract}

\section{Introduction}
Structural relationships between the renewal function and the underlying distribution are of great interest in renewal theory.  For a renewal process with decreasing failure rate (DFR) inter-renewal times, it is known that the renewal function is concave (Brown 1980).  Conversely, Brown (1981) conjectures that DFR inter-renewal times are also necessary for the concavity of the renewal function.  As shown by Shanthikumar (1988), there exist counter-examples to a discrete analogue of this conjecture.  Brown's conjecture in the continuous case, however, has remained open.  See Szekli (1986, 1990), Hansen and Frenk (1991), Shaked and Zhu (1992), Kijima (1992), and Kebir (1997) for related results and discussions.  Also relevant is the work of Lund, Zhao and Kiessler (2006), who use hazard rates and renewal sequences to study reversible Markov chains. 

In this note we construct absolutely continuous distributions that do not have decreasing failure rates but nevertheless lead to concave renewal functions.  That is, we give a definite answer to Brown's question in the negative.  Our counter-examples have the following feature.  On $[0, t_1]$ for some $t_1>0$, the inter-renewal time has a decreasing failure rate; on $[t_1, t_2]$ for some $t_2> t_1$, the failure rate strictly increases before decreasing again on $[t_2, \infty)$.  It is shown that, for a suitable class of such distributions, if the increase in failure rate on $[t_1, t_2]$ is small enough, and the decrease shortly after $t_2$ is fast enough, then the resulting renewal density is decreasing, i.e., the renewal function is concave.  Section~2 presents the precise statements and illustrates with a numerical example.  Section~3 contains the proofs. 

The renewal process is closely related to compound geometric random variables.  In Section~4, by adapting the arguments of de Bruijn and Erd\"{o}s (1953), we give an alternative proof of Shanthikumar's (1988) result that the DFR property is closed under geometric compounding. 

\section{Concavity of the renewal function}
Let $F(t)$ be a distribution function on $\mathbf{R}_+=[0, \infty)$ with $F(0)=0$.  Then the renewal function $M(t)$, i.e., the average number of renewals in $[0, t]$, for a renewal process with underlying distribution $F$ is given by
$$M(t)=F(t)+\int_0^t M(t-x)\, {\rm d} F(x),\quad t\geq 0.$$
(Some authors define $M(t)+1$ as the renewal function; our results work with either definition.)  If $F(t)$ is absolutely continuous with density $f(t)$, then so is $M(t)$, and a version of its density, $m(t)$, satisfies 
\begin{equation}
\label{density}
m(t)=f(t)+\int_0^t m(x) f(t-x)\, {\rm d} x,\quad t\geq 0.
\end{equation}
A positive function $g(x),\ x\in \mathbf{R}_+,$ is log-convex if $\log g(x)$ is convex on $\mathbf{R}_+$.  A distribution on $\mathbf{R}_+$ has DFR (decreasing failure rate), if its survival function is log-convex on $\mathbf{R}_+$.  We recall two fundamental results relating $M(t)$ to $F(t)$. 

\begin{theorem}[de Bruijn and Erd\"{o}s (1953); Brown (1980); Hansen and Frenk (1991)]
\label{thm0}
We have 
\begin{enumerate}
\item
If $F(t)$ has a log-convex density $f(t)$, then the renewal density $m(t)$ as in (\ref{density}) is also log-convex. 
\item
If $F(t)$ is DFR, then $M(t)$ is concave. 
\end{enumerate}
\end{theorem}
The question raised by Brown (1981) may be formulated as follows.
\begin{conjecture}
\label{conj}
If the renewal function $M(t)$ is concave on $\mathbf{R}_+$, then $F(t)$ is DFR.
\end{conjecture}
Shanthikumar (1988) resolves a discrete version of this conjecture by constructing a counter-example using auxiliary results on discrete Markov chains.  It has also been noted that the discrete example does not generalize and the continuous case is still open.  Our main result (Proposition \ref{thm1}) finally disproves Conjecture~\ref{conj}. 

\begin{proposition}
\label{thm1}
Let $0<t_1<\infty$.  Let $f(t)$ be a density function that is positive on $\mathbf{R}_+$ and continuously differentiable on each of $I_k,\ 0\leq k\leq 3,$ where $I_0=[0, t_1],\ I_1=[t_1, t_2],\ I_2=[t_2, t_3],\ I_3=[t_3,\infty)$, with $t_2, t_3$ to be determined.  That is, $f(t)$ is continuous on $\mathbf{R}_+$, but $f'(t)$ may jump at $t_k,\ k=1,2,3.$
Assume the corresponding hazard rate function $r(t)$ satisfies 
\begin{itemize}
\item[(i)]
$r'(t)<0,\ t\in I_0$; 
\item[(ii)]
on $I_1$ we have 
\begin{equation}
\label{I1}
r(t)=\frac{\lambda}{1 -\epsilon e^{\lambda t}},\quad t_1\leq t\leq t_2,
\end{equation}
for some $\epsilon\in (0, 1)$ where $\lambda>0$ is determined by $\epsilon$ and $r(t_1)$; 
\item[(iii)]
on $I_2$ we have 
\begin{equation}
\label{rpm}
r'(t)\leq r^2(t)-f(0)r(t),\quad t_2<t<t_3;
\end{equation}
\item[(iv)]
$r(t_3)\leq r(t_1)$ and $r'(t)\leq 0,\ t\in I_3$.  
\end{itemize}
Then for small enough $\epsilon>0$ and $t_2-t_1>0$, both depending on the specification of $r(t)$ for $t\in I_0$ only, 
the renewal density $m(t)$ given by (\ref{density}) decreases on $\mathbf{R}_+$. 
\end{proposition} 

Note that $r(t)$ strictly increases on $[t_1, t_2]$.  Proposition \ref{thm1} therefore settles Conjecture \ref{conj} in the negative.  An example of a survival function $\bar{F}$ satisfying Conditions (i) and (ii) is 
\begin{equation}
\label{Fbar}
\bar{F}(t) =\begin{cases} \frac{1}{2}( e^{- t} + 1 ), & 0\leq t \leq t_1, \\
\alpha e^{-\lambda t} -\beta, & t_1< t\leq t_2, \\
\end{cases}
\end{equation}
where $\beta>0$, and $\alpha$ and $\lambda$ are determined by $\beta$ via 
$$\bar{F}(t_1+)=\bar{F}(t_1)\quad {\rm and}\quad \bar{F}'(t_1+)=\bar{F}'(t_1-).$$
Specifically, 
$$\lambda = \left[1+(1+2\beta) e^{t_1}\right]^{-1}, \quad \alpha = (2\lambda)^{-1} e^{(\lambda -1)t_1}.$$
The $\epsilon$ in (\ref{I1}) corresponds to $\beta/\alpha$.  Condition (iii) says that the hazard rate should decrease fast shortly after $t_2$.  An example based on (\ref{Fbar}) that satisfies this condition is 
$$r(t)=r(t_2) e^{(t_2-t)/2},\quad t_2 < t \leq t_3,$$
which leads to 
\begin{equation}
\label{barF}
\bar{F}(t) = \bar{F}(t_2) \exp\left[-2r(t_2) \left(1-e^{(t_2-t)/2}\right)\right],\quad t_2<t\leq t_3.
\end{equation}
For Condition (iv), we need $t_3\geq t_2 + 2\log (r(t_2)/r(t_1))$ to ensure $r(t_3)\leq r(t_1)$, but $r(t)$ can stay flat on $t\in I_3$, which gives
\begin{equation}
\label{F3}
\bar{F}(t)=\bar{F}(t_3) e^{r(t_3)(t_3-t)},\quad t>t_3.
\end{equation}

As an illustration, Figure 1 shows the survival function, density, hazard rate, and renewal function for a distribution as specified by (\ref{Fbar}), (\ref{barF}) and (\ref{F3}) with $t_1=1,\ t_2=1.5,\ t_3=2,$ and $\beta=0.02$.  The almost imperceptible decrease of $m(t)$ on $t\in [1, 1.5]$ is verified numerically as Proposition~\ref{thm1} guarantees monotonicity of $m(t)$ for small enough $\beta>0$ and $t_2-t_1>0$ but does not specify how small $\beta$ or $t_2-t_1$ has to be. 

\begin{figure}
\psfrag{Fbar(t)}{\small $\bar{F}(t)$}
\psfrag{r(t)}{\small $r(t)$}
\psfrag{f(t)}{\small $f(t)$}
\psfrag{m(t)}{\small $m(t)$}
\psfrag{t}{\small $t$}
\begin{center}
\includegraphics[width=3.5in, height=5.2in, angle=270]{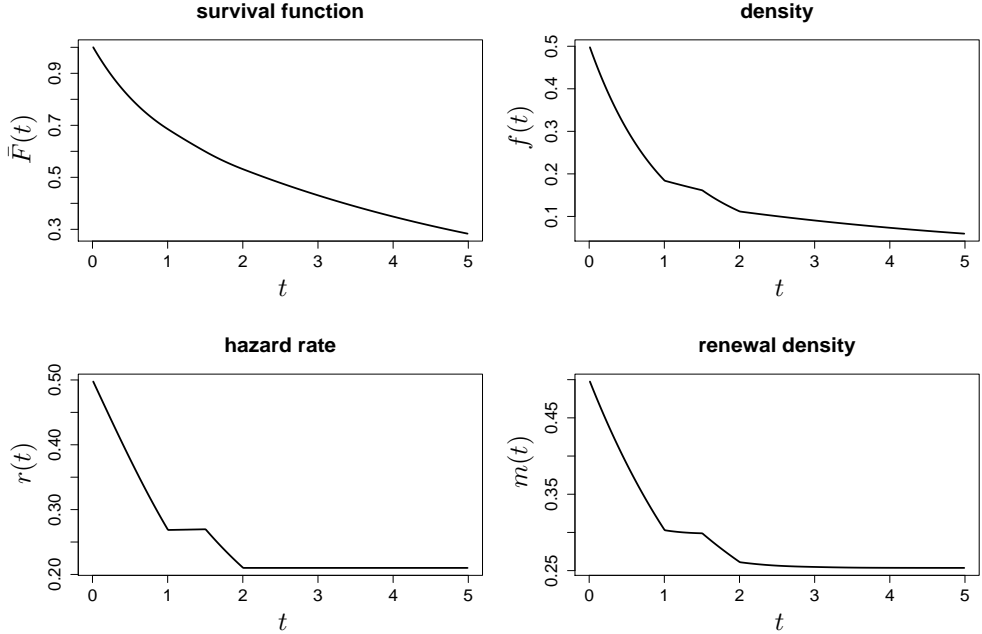}
\end{center}
\caption{Illustration of a counter-example given by (\ref{Fbar}), (\ref{barF}) and (\ref{F3}) with $t_1=1,\ t_2=1.5,\ t_3=2$ and $\beta=0.02$.}
\end{figure}

\section{Proof of Proposition \ref{thm1}}
We first establish a simple but useful identity. 
\begin{lemma}
\label{lem}
Let $r(t)$ and $\bar{F}(t)$ denote the hazard rate and survival function, respectively, for a distribution with density $f(t)$ on $\mathbf{R}_+$.  Assume $f(t)$ is absolutely continuous and $f'(t)$ is bounded on every compact sub-interval of $\mathbf{R}_+$.  Then the renewal density $m(t)$ as defined by (\ref{density}) satisfies 
\begin{equation}
\label{key}
m'(t)= r'(t)\bar{F}(t) +\int_0^t m'(x)[r(t-x)-r(t)]\bar{F}(t-x)\, {\rm d}x,\quad t>0.
\end{equation}
\end{lemma}
A discrete version of (\ref{key}) can be traced back to Kaluza (1928). 
\begin{proof}[Proof of Lemma \ref{lem}]
The conditions guarantee that $m(t)$ is absolutely continuous.  In fact, we may differentiate under the integral sign in (\ref{density}) and get 
\begin{equation}
\label{mtp0}
m'(t) = f'(t) + m(t) f(0) + \int_0^t m(x) f'(t-x)\, {\rm d} x,\quad t>0.
\end{equation}
Integration by parts then yields  
\begin{equation}
\label{mtprime}
m'(t) = f'(t) + m(0) f(t) + \int_0^t m'(x) f(t-x)\, {\rm d} x,\quad t>0.
\end{equation}
We also have 
\begin{align}
\nonumber
\int_0^t m'(x) \bar{F}(t-x)\, {\rm d}x &= m(t) - f(0)\bar{F}(t) -\int_0^t m(x) f(t-x)\, {\rm d}x\\
\label{int}
&=f(t)-f(0)\bar{F}(t),
\end{align}
where the first step uses integration by parts and the second uses (\ref{density}).  The identity (\ref{key}) follows by expanding its right hand side and applying (\ref{int}) and (\ref{mtprime}) to simplify. 
\end{proof}

\begin{proof}[Proof of Proposition \ref{thm1}]
Since $f'(t)$ is piecewise continuous, so is $m'(t)$ as seen from (\ref{mtp0}).  We have $m'(0+)=r'(0+)< 0$.  Suppose $m'(t)$ ever becomes nonnegative on $I_0$.  Then letting $t_*$ be the smallest $t\in (0, t_1)$ such that $m'(t)\geq 0$ we have $m'(x)<0,\ 0<x<t_*$, and by Condition (i) $r'(t_*)<0,\ r(t_*-x)-r(t_*) >0,\ 0<x<t_*$.  It follows from (\ref{key}) that $m'(t_*)< 0$, a contradiction.  Thus $m'(t)<0,\ t\in I_0$.  In fact, applying (\ref{key}) again yields 
\begin{equation}
\label{strong}
m'(t)< r'(t)\bar{F}(t),\quad t\in I_0,
\end{equation}
where the left derivatives are used if $t=t_1$.  

By (\ref{I1}) we have 
$$r'(t)=(\lambda^{-1} r(t)-1)r(t),\quad t_1<t<t_2,$$
where $\lambda$ is determined from $\epsilon$ via $r(t_1)=\lambda /(1 -\epsilon e^{\lambda t_1})$.  For fixed $r(t_1)$ as $\epsilon\downarrow 0$ we have $\lambda \uparrow r(t_1)$ and hence $r'(t_1+)\downarrow 0$.  Denoting $\delta =m'(t_1-) - r'(t_1-)\bar{F}(t_1)$, and noting $\delta<0$ by (\ref{strong}), we get 
\begin{align*}
m'(t_1+) = r'(t_1+)\bar{F}(t_1)+\delta <0, 
\end{align*}
for small enough $\epsilon>0$. 
Because $m'(t)$ is continuous on $(t_1, t_2)$, and $m'(t_1+)<0$, we have $m'(t)<0,\ t\in (t_1, t_2),$ if $t_2 - t_1$ is small enough.  Thus $m(t)$ decreases on $I_1$. 

Also, $m(t)$ must strictly decrease on $I_2$.  Assume the contrary and let $t^*$ be the smallest $t\in [t_2, t_3)$ such that $m'(t+)\geq 0$.  Then (\ref{mtprime}) gives 
\begin{equation}
\label{mp+}
0\leq m'(t^*+)< f'(t^*+)+m(0)f(t^*),
\end{equation}
because inside the integral $m'(x)<0,\ x\in [0, t^*)$.  However, by (\ref{rpm}) we have 
$$f'(t^*+)+ m(0)f(t^*) = \bar{F}(t^*)\left[r'(t^*+)-r^2(t^*)+f(0)r(t^*)\right]\leq 0,$$
which contradicts (\ref{mp+}). 

Finally, we show that $m(t)$ decreases on $I_3$ by applying (\ref{key}) again.  The assumptions $r(t_3)\leq r(t_1)$ and $r'(t)\leq 0,\ t\in I_3,$ ensure that $r(t-x)\geq r(t),\ 0<x<t,\ t>t_3,$ with strict inequality if $t-x<t_1$.  It is already shown that $m'(x)< 0$ for $x<t_3$.  Thus (\ref{key}) implies $m'(t_3+)< 0$.  The same argument proving $m'(t)<0$ for $t\in I_0$ then shows that $m(t)$ decreases on $I_3$. 
\end{proof}

\section{Preservation of DFR under geometric compounding}
Compound geometric random variables appear naturally in areas such as queuing theory (see, e.g., Szekli 1986) and financial risk modeling.  It is well known that log-convexity is closed under geometric compounding (this is essentially Part 1 of Theorem \ref{thm0}).  Shanthikumar (1988) showed that the DFR property is also closed under geometric compounding.  This was achieved by establishing auxiliary results on discrete Markov chains.  It may be worthwhile to note that the argument of de Bruijn and Erd\"{o}s (1953) can be adapted to give a short proof of Shanthikumar's (1988) result (Part 1 of Theorem \ref{thm2}).  The same argument yields a parallel result (Part 2 of Theorem \ref{thm2}) concerning the increasing failure rate (IFR) property. 

\begin{theorem}
\label{thm2}
Let $X$ be a random variable on $\mathbf{N}=\{1, 2, \ldots\}$ and let $T$ be a geometric with parameter $p\in (0, 1)$, i.e., $\Pr(T=n)=p q^{n-1},\ n=1, 2, \ldots,\ q\equiv 1-p$.  Define the random sum $Y\equiv \sum_{k=1}^T X_k$ where $X_k$ are independent (and also independent of $T$) and identically distributed as $X$. 
\begin{enumerate}
\item
If $\log\Pr(X\geq n)$ is convex in $n\in \mathbf{N}$, i.e., $X$ is discrete DFR, then so is $Y$.
\item
If $\log\Pr(Y\geq n)$ is concave in $n\in \mathbf{N}$, i.e., $Y$ is discrete IFR, then so is $X$.
\end{enumerate}
\end{theorem}

\begin{proof}
Denote 
$$f_n=\Pr(X=n),\quad \bar{F}_n=\Pr(X\geq n), \quad g_n=\Pr(Y=n),\quad \bar{G}_n=\Pr(Y\geq n).$$
We have the recursions 
\begin{equation}
\label{recursion}
g_n= pf_n + q \sum_{k=1}^{n-1} f_k g_{n-k},\quad \bar{G}_n = \bar{F}_n +q\sum_{k=1}^{n-1} f_k \bar{G}_{n-k},\quad n=1,2,\ldots.
\end{equation}
The following identity is analogous to Equation (7) of de Bruijn and Erd\"{o}s (1953); Hansen (1988) uses similar identities for compound Poisson probabilities (see Yu 2009 for related work).  It is proved by expanding the right hand side and then applying (\ref{recursion}). 
\begin{align}
\nonumber
\bar{F}_n (\bar{G}_{n+2}\bar{G}_n -\bar{G}_{n+1}^2) = & p \bar{G}_n (\bar{F}_{n+2}\bar{F}_n -\bar{F}_{n+1}^2) \\
\label{induct}
&+q\sum_{k=2}^n (\bar{F}_{n+1} f_{k-1} -\bar{F}_n f_k)(\bar{G}_{n+1}\bar{G}_{n+1-k} -\bar{G}_n \bar{G}_{n+2-k}).
\end{align}
In particular, $\bar{G}_3\bar{G}_1- \bar{G}_2^2=p(\bar{F}_3\bar{F}_1 -\bar{F}_2^2).$  Assuming $\bar{F}_n$ is log-convex, we get $\bar{G}_{n+1}\bar{G}_{n+1-k} \geq \bar{G}_n \bar{G}_{n+2-k},\ 2\leq k\leq n,$ and $\bar{G}_{n+2}\bar{G}_n \geq \bar{G}_{n+1}^2,\ n\geq 1,$ by induction from (\ref{induct}).  Thus $\bar{G}_n$ is log-convex in $n\in \mathbf{N}$, i.e., $Y$ is discrete DFR, and Part 1 is proved.  Similarly, assuming $\bar{G}_n$ is log-concave, we get $\bar{F}_{n+1} f_{k-1} \leq \bar{F}_n f_k,\ 2\leq k\leq n,$ and $\bar{F}_{n+2}\bar{F}_n \leq \bar{F}_{n+1}^2,\ n\geq 1,$ by induction.  Thus $\bar{F}_n$ is log-concave in $n\in \mathbf{N}$, i.e., $X$ is discrete IFR, and Part 2 is proved.
\end{proof}

\section*{Acknowledgement}
The author would like to thank Mark Brown, George Shanthikumar and Ryszard Szekli for their helpful comments.

\end{document}